\documentclass[10pt]{article}
\usepackage{xcolor}
\usepackage{graphicx}
 \usepackage{amsfonts}
 \usepackage{amsmath}
 \usepackage{amsthm}
 \usepackage{natbib} 
\newtheorem{theorem}{Theorem}
\newtheorem{corollary}{Corollary}
\newtheorem{lemma}{Lemma}
\newtheorem{proposition}{Proposition}
\newtheorem{remark}{Remark}

\def\process#1{#1=\{#1_t\}_{t\geq 0}}
\newcommand\var{\operatorname{\mathbf{var}}}
\newcommand{\indicator}{\mathbf{1}}
\newcommand{\F}{\mathcal{F}}

\newcommand{\R}{\mathbb{R}}
\newcommand{\M}{\mathcal{M}}
\newcommand\probz{{\operatorname{\mathbf{P}}}}
\newcommand\proba{{\operatorname{\mathbf{P}}_x}}
\newcommand\E{{\operatorname{\mathbf{E}}_x}}
\newcommand\Ez{{\operatorname{\mathbf{E}}_0}}
\newcommand\Et{{\operatorname{\tilde{\mathbf{E}}}_0}}

\begin{document}
 \title{Optimal stopping for L\'evy processes \\ with polynomial rewards}
 \author{Ernesto  Mordecki\footnote{Universidad de la Rep\'ublica, Facultad de Ciencias, Centro de Matem\'atica. Igu\'a 4225, 11400 Montevideo, Uruguay. e-mail: mordecki@cmat.edu.uy}\qquad 
Yuliya Mishura
 \footnote{National Taras Shevchenko University of Kyiv,
Department of Probability Theory, Statistics and Actuarial Mathematics.
Volodymirska 60, 01601 Kyiv, Ukraine. e-mail: myus@univ.kiev.ua.}
 }
 \maketitle
 \begin{abstract}
Explicit solution of an infinite horizon optimal stopping problem 
for a L\'evy processes with a polynomial reward function is given,
in terms of the overall supremum of the process, when the solution of the problem is one-sided.
The results are obtained via the generalization of known results about the \emph{averaging function} associated with the problem. This averaging function can be directly computed in case of polynomial  rewards.
To illustrate this result, examples for general quadratic and cubic polynomials are discussed in case the process is Brownian motion, and the optimal stopping problem for a quartic polynomial and a Kou's process is solved.
\end{abstract}

{Keywors: Optimal stopping, L\'evy processes, polynomial rewards}

{AMS MSC: 60G40}

\section{Introduction}
Since the seminal work of \cite{darling-et-al:1972},
giving the solution to the optimal stopping problem
for random walks, and reward functions of  the form
$g(x)=x^+$ and $g(x)=(e^x-1)^+$, in terms of the
distribution of the maximum of the random walk, it became clear the possibility of linking
these two relevant problems in probability theory:
the optimal stopping problem and the computation of the distribution of
the overall maximum of a random walk.
The natural question that this work posed was the possibility of extending these results
to more general classes of processes,
and to more general reward functions.


The first results for L\'evy processes were obtained by \cite{poli,bachelier},
where the similar corresponding problems for arbitrary L\'evy processes are solved,
based on a discretization approximation argument,
for the same reward functions, with the novelty of the consideration of the decreasing put reward
$g(x)=(K-e^x)^+$, that has a solution in terms of the overall infimum of the process.
The first results for general payoffs were obtained by \cite{bs:2002}.
Namely, using the technique of the Pseudo-Differential operators,
these authors obtained solutions to optimal stopping problems considering a large
class of reward functions, making clear that the obtained previously results were
not based on particular properties of the payoff function, but only on the properties of the
L\'evy processes. Their approach is analytic, based on the decomposition of an operator,
that is in certain sense equivalent to the Wiener-Hopf factorization, and imposes certain restrictions on the class of
L\'evy processes to which the results can be applied. For a general exposition of these results see also
\cite{bl-book}.
Afterwards, \cite{novikov-shiryaev:2004}
solved the optimal stopping problem for arbitrary random walks
and reward functions of the form $g(x)=(x^+)^n$, in terms
of the Appell polynomials, and \cite{novikov-shiryaev:2007} gave the solution 
to the problem with a power function reward with real and positive exponent, for both random walks
and L\'evy process.
 \cite{salminen} applies the representation method
for this problem (initiated in \cite{salminen85}) finding the representing measure of
the value function.
More recently, \cite{mt} considered the optimal stopping problem for
a general polynomial reward and a random walk.
\cite{alili-kyprianou:2005}
and \cite{kyprianou-surya:2005}
obtained a new proof of the main results in \cite{poli}
and a generalization of the results in \cite{novikov-shiryaev:2004}
for L\'evy process respectively, in both cases based on the strong
Markov property of L\'evy processes.
These contributions were summarized in the monograph by \cite{kyprianou:2006}.
On the way to the  consideration of more general processes,
\cite{ms} obtained a representation of the value function for Hunt processes, that in the
case of L\'evy processes give a representation in terms of the maximum of the process,
and \cite{bcs}  exploited the excessive property of the maximum of a Markov process
to obtain a verification theorem. It became then clear that the results were based on
the probabilistic properties which random walks and
L\'evy processes share, i.e. the independence
and  homogeneity of increments, and not on the particular
form of the reward functions. Nevertheless, some particular reward functions
admitted solutions in closed form.

The approach that we use in this paper is the \emph{averaging problem}, that was introduced in \cite{surya:2007} (see also \cite{surya:2007t}). 
%
The objective of the present paper is then twofold.
We first present a theorem that summarizes and slightly improves the results of
\cite{surya:2007} and
\cite{bcs} in the case of L\'evy processes.
The improvement consists in the observation that the averaging function in \cite{surya:2007}
(or the function $\hat{f}$ in \cite{bcs}) need not to be defined in the whole line, consequently the
condition of this function to be negative on a certain set is not necessary.
This allows to apply the result to larger classes of payoffs functions, what can be verified for certain polynomial rewards
(see Remark \ref{remark:1}).
The second objective of the paper is to apply the previous results to the class of general polynomial rewards.
The main result there is a simple algorithm to compute the averaging polynomial $P$ of a given polynomial $p$.

The content of the paper is as follows.
In Section 2 we formulate the problem and prove the main results.
In Section 3 we specializes to polynomial rewards.
In Section 4 we present some examples: we discuss in detail the optimal stopping problem for
Brownian motion and general quadratic and cubic polynomials, and also solve explicitely the optimal stopping
problem for a quartic polynomial for a Kou process.
 

\section{Formulation of the problem and main results}


Let $\process X$ be a L\'evy process defined on a
stochastic basis ${\cal B}=(\Omega, {\cal F}, {\bf F}=({\cal
F}_t)_{t\geq 0}, \proba)$ departing from $X_0=x$.  For $z\in i\R$, the L\'evy-Khintchine formula states
$
\Ez e^{zX_t}=e^{t\psi(z)}
$
with
\begin{equation}\label{eq:char-exp}
\psi(z)=az+{\sigma^2\over 2}z^2+\int_{\R}\left(e^{zy}-1-zh(y)\right)\Pi(dy),
\end{equation}
where $a\in\R$, $\sigma\geq 0$ and $\Pi(dy)$ that satisfies $\int_{\R}(1\wedge y^2)\Pi(dy)<\infty$
conform the characteristic triplet
$(a,\sigma,\Pi)$ of the process.
Here $h(y)=y\indicator_{\{|y|<1\}}$ is a truncation function.
Given the stochastic basis ${\cal B}$ the set of stopping times is the set of random variables
$$
\M=\{\tau\colon\Omega\to[0,\infty] \text{ such that } \{\tau\leq t\}\in\mathcal{F}_t \text{ for all $t\geq 0$}\}.
$$
Observe that we allow the possibility $\tau=\infty$, as for several optimal stopping problems, the optimal stopping time
is within this class.
A key r\^ole in the solution of the problem is played by the overall maximum of the process,
defined, for $r\geq 0$ by
$$
M=\sup\{X_t\colon 0\leq t\leq e(r)\},
$$
where $e(r)$ is an exponential random variable of parameter $r>0$, and we assume $e(0)=\infty$.
We further assume thorough the paper that $M$ is a proper random variable.
This entails either that  $r>0$ or that $\process{X}$ drifts to $-\infty$ when $r=0$, and that
$\proba(M>x)>0$, excluding the case of the negative of a subordinator, that gives $M=x$ a.s.

Given a non-negative payoff function $g(x)$,
a process $\{X_t\}_{t\geq 0}$ departing from $X_0=x$ adapted to a filtration $\mathbf{F}$,
and a discount factor $r\geq 0$, the optimal stopping problem consists
in finding the value function ${V}(x)$ and the optimal stopping rule $\tau^*$ such that
\begin{equation}\label{eq:osp}
{V}(x)=\sup_{\tau\in\M}\E e^{-r\tau}g(X_{\tau})=\E e^{-r\tau^*}g(X_{\tau^*}).
\end{equation}
Following \cite{s} we assume that the payoff received in the set 
$\{\omega\colon\tau(\omega)=\infty\}$ is 
$$
\limsup_{t\to\infty}e^{-rt}g(X_t).
$$
In the present paper we are interested in problems with one-sided solution, i.e. such that the optimal stopping rule is of the form
\begin{equation}\label{eq:taustar}
\tau^*=\inf\{t\geq 0\colon X_t\geq x^*\},
\end{equation}
for some critical threshold $x^*$. For this reason we assume that $\limsup_{x\to\infty}g(x)=\lim_{x\to\infty}g(x)=0$.

The averaging problem for optimal stopping, introduced by \cite{surya:2007},
consists in finding an auxiliary function $Q$ such that
\begin{equation}\label{eq:average}
\E Q(M)=g(x)\text{ for all $x\in\R$},
\end{equation}
where $g$ is the payoff function of the problem and $M$ the overall maximum.
This approach, combined with the strong Markov property and invariance of increments of L\'evy process
gives a fluctuation identity that allows to write the value function of the problem
\eqref{eq:osp} in terms of $M$ (see \eqref{eq:threshold} in Lemma 1 below).
Here we present a generalization of the results in \cite{surya:2007}.

\begin{theorem}\label{theorem:one}
Consider a L\'evy process $\process X$,
a discount rate ${r}\geq 0$,
and a reward function $g\colon\R\to[0,\infty)$ such that $\lim_{x\to-\infty}g(x)=0$.
Assume that there exists a point $x^*$ and a non-decreasing function $G^*:[x^*,\infty)\to\R$ such that
$$
\E {G^*}({M})=g(x), \text{ for all $x\geq x^*$.}
$$
Define the function
\begin{equation}\label{eq:ge}
G(x)=
\begin{cases}
G^*(x),	& \text{ if $x\geq x^*$},\\
0,			& \text{ if $x<x^*$},
\end{cases}
\end{equation}
and the function $V\colon\R\to\R$ by
\begin{equation}\label{eq:ve}
V(x)=\E {G}({M}),\quad
\text{for all $x\in\R$.}
\end{equation}
If the condition
\begin{equation}\label{eq:relevant}
V(x)\geq g(x),\quad\text{for all $x< x^*$},
\end{equation}
is satisfied, then the optimal stopping problem \eqref{eq:osp}
has value function $V(x)$ in \eqref{eq:ve}, and \eqref{eq:taustar} is an optimal stopping time for the problem.
\end{theorem}
\begin{remark}\label{remark:1}
Compared to Theorem 5.3.1. in \cite{surya:2007t}, Theorem \ref{theorem:one} above does not require the solution of the averaging problem for $g$ and $M$ to be found on the whole real line, 
but only on a certain set of the form $[x^*,\infty)$. 
The relevant new condition to be verified on the set $(-\infty,x^*)$ is \eqref{eq:relevant}. 
If the averaging function $Q$ (satisfying \eqref{eq:average}) can be defined in the whole real line 
and it satisfies  $Q(x)\leq 0$ on the set $(-\infty,x^*)$, then, condition 
\eqref{eq:relevant} follows (see Corollary \ref{corollary:1}). Our function $G$ is simply defined to be zero on this set.
In Example \ref{example:1} when $a=-1$ we observe that \eqref{eq:relevant} is verified while the averaging function corresponding to \eqref{eq:average}
takes positive values (for instance, $P_2(0)=1$, see Figure \ref{figure:1}).
Furthermore, condition \eqref{eq:relevant} is slightly more general than condition
(b)(ii) in Theorem 2.4 in \cite{bcs}.
\end{remark}
As usual in optimal stopping proofs we have to verify two statements:
\begin{align}
V(x)&=\E e^{-r\tau^*}g(X_{\tau^*})=\E e^{-r\tau^*}g(X_{\tau^*}).\label{eq:equality}\\
V(x)&\geq\E e^{-r\tau}g(X_{\tau})=\E e^{-r\tau}g(X_{\tau}),
\quad\forall\tau\in\M.\label{eq:inequality}
\end{align}
These two statements are proved based on the following two  lemmas  which proofs  follow essentially
the respective proofs of \cite{surya:2007} and \cite{bcs} with the minor necessary modifications.
\begin{lemma}\label{lemma:one}
Consider a L\'evy process $X$,
a discount rate ${r}\geq 0$,
a reward function $g$, 
a threshold $x^*$, an averaging function $G^*$, 
and the extended function $G$, all this elements as in Theorem \ref{theorem:one}.
Then, for any $a\geq x^*$ and $x\in\R$,
\begin{equation}\label{eq:threshold}
\E {G}({M})\indicator_{\{M\geq a\}}=\E e^{-{r}\tau_a}g(X_{\tau_a})\indicator_{\{\tau_a<\infty\}}.
\end{equation}
In particular, when $a=x^*$, for $\tau^*$ in \eqref{eq:taustar} and $r>0$, we have 
\begin{equation*}
\E {G}({M})\indicator_{\{M\geq x^*\}}=\E e^{-{r}\tau^*}g(X_{\tau^*})\indicator_{\{\tau^*<\infty\}}.
\end{equation*}
meanwhile, when $r=0$, we have
\begin{equation*}
\E {G}({M})\indicator_{\{M\geq x^*\}}=\E g(X_{\tau^*})\indicator_{\{\tau^*<\infty\}}.
\end{equation*}
\end{lemma}
\begin{proof}
Consider, for  $a\geq x^*$, a hitting time of the form
\[
\tau_a=\inf\{t\geq 0\colon {}X_t\geq a\}.
\]
As L\'evy processes satisfy  the homogeneity property of increments in time and space, 
conditionally to the $\sigma$-algebra $\F_{\tau_a}$, and on the set $\{\tau_a<\infty\}$, the process
$\tilde{X}_s=X_{\tau_a+s}-X_{\tau_a}$ is independent of $\F_{\tau_a}$ and has the same distribution as $X$ (see Theorem 7, Chapter 4 in \cite{skorokhod}).
We then consider two independent L\'evy processes $X$ and $\tilde{X}$ defined on a product probability space
$\probz\times\tilde{\probz}$. Consider first the case $r=0$. We have
\begin{align*}
\E G(M)&\indicator_{\{M\geq a\}}=
\E {G}\Big(\sup_{0\leq t<\infty}X_t\Big)\indicator_{\{\tau_a<\infty\}}
\\&
=
\E G\Big(X_{\tau_a}+\sup_{\tau_a\leq t<\infty}(X_t-X_{\tau_a})\Big)\indicator_{\{\tau_a<\infty\}}
\\&=
\E\Et G\Big(X_{\tau_a}+\sup_{0\leq s<\infty}\tilde{X}_s\Big)\indicator_{\{\tau_a<\infty\}}=
\E g\left(X_{\tau_a}\right)\indicator_{\{\tau_a<\infty\}}.
\end{align*}
We proceed now for $r>0$.
In this case, we have
\begin{align*}
\E {G}({}{M})&\indicator_{\{{}M\geq a\}}=
\E {G}\Big(\sup_{0\leq t<{e(r)}}X_t\Big)\indicator_{\{\tau_a<{e(r)}\}}\\
&=
\E {G}\Big({}X_{\tau_a}+\sup_{\tau_a\leq t<{e(r)}}(X_t-X_{\tau_a})\Big)\indicator_{\{\tau_a<{e(r)}\}}\\
&=
\E \int_{\tau_a}^{\infty}
{G}\Big({}X_{\tau_a}+\sup_{0\leq s<t-\tau_a}(X_{\tau_a+s}-X_{\tau_a})\Big)
{r}e^{-{r}t}dt\indicator_{\{\tau_a<\infty\}} \\
&=
\E
e^{-{r}\tau_a}
\int_0^{\infty}
{G}\Big({}X_{\tau_a}+\sup_{0\leq s<v}(X_{\tau_a+s}-X_{\tau_a})\Big){r}e^{-{r}v}dv\indicator_{\{\tau_a<\infty\}}\\
&=
\E\Et
e^{-{r}\tau_a}
\int_0^{\infty}
{G}\Big({}X_{\tau_a}+\sup_{0\leq s<v}\tilde{X}_s\Big){r}e^{-{r}v}dv\indicator_{\{\tau_a<\infty\}}\\
&=
\E
e^{-{r}\tau_a}
\Et G\big(X_{\tau_a}+\tilde{{M}}\big)\indicator_{\{\tau_a<\infty\}}=
\E
e^{-{r}\tau_a}
g({}X_{\tau_a})\indicator_{\{\tau_a<\infty\}},
\end{align*}
concluding the proof.
\end{proof}
\begin{remark}Fluctuation identities as the one presented in the previous Lemma in case of exponential or related to exponential functions
have been obtained by
\cite{darling-et-al:1972} for random walks 
and by \cite{alili-kyprianou:2005} for L\'evy processes.
In case of power functions with positive integer exponent \cite{novikov-shiryaev:2004} introduced the Appel polynomials to obtain
similar identities for random walks, and \cite{kyprianou-surya:2005} obtained the corresponding result for
L\'evy processes. The case of power functions with real positive exponent was considered in \cite{novikov-shiryaev:2007}
for both random walks and L\'evy processes.
The identity for general functions was obtained by \cite{surya:2007t}, see also \cite{surya:2007}.
\end{remark}

\begin{lemma}\label{lemma:two}
Consider a non-negative non-decreasing function $f(x)$ and  a real $r\geq 0$.
Then:
(a)
The function
$
h(x)=\E f({M})\quad (x\in\R)
$
is $r$-excessive, and, in consequence,
(b) the process
$\{e^{-{r}t}h(X_t)\}$ is a supermartingale.
\end{lemma}
\begin{proof}
The fact that (b) follows from (a) is standard, see for
example \cite{s}.
The statement (a) is a corollary of Lemma 2.2  in \cite{bcs}, as for non-decreasing $f$
we have
$$
\sup_{0\leq t\leq e(r)}f(X_t)=f\left(\sup_{0\leq t\leq e(r)}X_t\right)=f(M),
$$
concluding the proof.
\end{proof}
\begin{proof}[Proof of the Theorem \ref{theorem:one}]
We finally observe that \eqref{eq:equality} follows from Lemma \ref{lemma:one} with $a=x^*$.
To prove \eqref{eq:inequality} we observe that $V(x)$ is excessive based on Lemma \ref{lemma:two} applied to the non-decreasing function $G(x)$ in \eqref{eq:ge}, so, for any stopping time $\tau\in\M$, we have
$$
V(x)\geq \E e^{-r\tau}V(X_\tau)\geq \E e^{-r\tau}g(X_\tau).
$$
This concludes the proof of Theorem \ref{theorem:one}.
\end{proof}
\begin{remark}
If the equality in \eqref{eq:relevant} holds for some $x<x^*$, then, defining the set
$$
S=\{x\in \R\colon V(x)=g(x)\},
$$
 the stopping time
$$
\tau^{**}=\inf\{t\geq 0\colon X_t\in S\}
$$
is also an optimal stopping time for the problem \eqref{eq:osp}. In fact, from the super-martingale property,
as $\tau^{**}\leq \tau^*$,
we have
$$
V(x)\geq \E e^{-r\tau^{**}}V(X_{\tau^{**}})\geq \E e^{-r\tau^{*}}V(X_{\tau^{*}}),
$$
obtaining that $\E e^{-r\tau^{**}}V(X_{\tau^{**}})=\E e^{-r\tau^{*}}g(X_{\tau^{*}})$.
\end{remark}

\begin{remark}
A method to find $G^*(x)$ and $x^*$ consists in first imposing condition \eqref{eq:average} for all $x\in\R$, i.e. in finding the averaging function of $g$ and $M$, and then finding its largest root.
This determines $G^*$ for $x\geq x^*$, in case it is a non-decreasing function on this half-line.
\end{remark}

The following result gives a sufficient condition in order to verify condition \eqref{eq:relevant}.
\begin{corollary}[\cite{surya:2007t}]\label{corollary:1}
Assume that there exists a function ${Q}\colon\R\to\R$ such that
\[
\E {Q}({}{M})=g(x)\text{ for all $x$,}
\]
and a real constant $x^*$ such that whenever
\[
x<x^*<y< z
\]
we have
\begin{equation}\label{eq:surya}
{Q}(x)\leq {Q}(x^*)=0\leq {Q}(y)\leq {Q}(z).
\end{equation}
Then $G^*(x)=Q(x)$ when $x\geq x^*$ verifies the conditions of Theorem \ref{theorem:one} ,
and $\lim_{x\to\infty}g(x)=0$.
\end{corollary}
\begin{proof}[Proof of the Corollary]
Let us check first that 
\begin{equation}\label{eq:gez}
\lim_{x\to-\infty}g(x)=0.
\end{equation}
 In fact, if $\limsup_{t\to-\infty}g(x)>0$ there exists a decreasing sequence $x_n\to-\infty$ such that $g(x_n)\geq\ell>0$.
 But
 \begin{align*}
 g(x_n)&=\Ez Q(x_n+M)\\&=\Ez Q(x_n+M)\indicator_{\{x_n+M\geq x^*\}}
 +\Ez Q(x_n+M)\indicator_{\{x_n+M< x^*\}}\\
 		&\leq \Ez Q(x_n+M)\indicator_{\{x_n+M\geq x^*\}} \to 0	\text{ as $n\to \infty$},
 \end{align*}
 by dominated convergence, as $Q(x)\indicator_{\{x+M\geq x^*\}}$ is decreasing in $x$ by hypothesis,
 giving a contradiction, and concluding \eqref{eq:gez}.

The rest of the proof is immediate as condition \eqref{eq:surya} implies condition \eqref{eq:relevant}.
In fact, for $G$ defined in \eqref{eq:ge}, we have
$$
V(x)=\E G(M)\geq \E Q(M)= g(x),
$$
concluding \eqref{eq:relevant}, and the proof of the Corollary. 
\end{proof}

\section{Polynomial rewards}
Our payoff function is constructed from a polynomial
\begin{equation}\label{eq:polynomial}
p_n(x)=x^n+a_{n-1}x^{n-1}+\cdots+a_1x,
\end{equation}
where we assume that $x=0$ is a root of $p_n(x)$.
The payoff is  the positive part of a polynomial, for positive values of the variable $x$:
$$
g(x)=\left(p_n(x^+)\right)^+=
\begin{cases}
p_n(x)^+,	&\text{when $x\geq 0$},\\
0,													&\text{otherwise.}
\end{cases}
$$
Observe that, the problem  \eqref{eq:osp} with reward function $\alpha g(\cdot+x_0)$ has solution
$\alpha V(\cdot+x_0)$, taking
the first coefficient $a_n=1$ and $x=0$ as the smallest root of $p_n$ in \eqref{eq:polynomial} entails no loss of generality for any polynomial with positive leading coefficient and at least one root.

\subsection{The averaging polynomial}
We search for a function ${P}_n(x)$ such that
\begin{equation}\label{eq:average-pol}
\E {P}_n({}M)=p_n(x), x\in \R.
\end{equation}
It is not difficult to see that this averaging function can be taken to be a polynomial of order $n$,
\begin{equation}\label{eq:averaging}
{P}_n(x)=x^n+b_{n-1}x^{n-1}+\cdots+b_1x+b_0.
\end{equation}
Assume that the first $n$  moments of $M$ are finite and denote them by $\mu_k=\Ez(M^k)\ (k=1,\dots,n)$. Denote
$\mu_0=1$. With this notation,
the l.h.s. in equation \eqref{eq:average-pol}, after changing the order in the sums,
reads
$$
\sum_{k=0}^nb_k\sum_{\ell=0}^kC^\ell_k x^\ell \mu_{k-\ell}=
\sum_{\ell=0}^n\left(\sum_{k=\ell}^nb_kC^\ell_k  \mu_{k-\ell}\right) x^\ell,
$$
that equating coefficients of equal degree in \eqref{eq:average-pol}
gives
\begin{equation}\label{eq:system}
\sum_{k=\ell}^nb_kC^\ell_k  \mu_{k-\ell}=a_\ell,\ \ell=n,n-1,\dots,0.
\end{equation}
This system of equations can be solved recursively backwards,
i.e.
$$
\aligned
b_n&=1,\\
b_{n-1}&=a_{n-1}-n\mu_1,\\
b_\ell&=a_\ell-\sum_{k=\ell+1}^nb_kC_k^
\ell\mu_{k-\ell},\quad\ell=n-2,\dots,0,
\endaligned
$$
where we put $a_0=0.$
\begin{remark} An equivalent way to obtain the averaging function $P_n(x)$ in \eqref{eq:average-pol}
is to write it as
$$
P_n(x)=\sum_{k=1}^n a_k Q_k(x)
$$
where the $Q_k(x)$ are the Appell polynomials of the random variable $M$,
introduced in \cite{novikov-shiryaev:2004}, applied also in \cite{kyprianou-surya:2005}, \cite{salminen} and \cite{mt}.
\end{remark}
We have the following simple result.
\begin{proposition}\label{proposition:aux} Consider a polynomial $p_n(x)$ as in \eqref{eq:polynomial}.
\par\noindent
{\rm (a)} The averaging polynomial $P_n(x)$ constructed as in \eqref{eq:averaging}
has at least one positive root.
\par\noindent
{\rm (b)} If $x^*$ denotes the largest root of $P_n(x)$, we have $p_n(x)\geq 0$ for $x\geq x^*$ and $p_n(x)> 0$ for $x> x^*$.
\end{proposition}
\begin{proof} (a)
If $P_n(x)$ has no positive root for $x\geq 0$ then $P_n(x)>0$ for all $x>0$.
As $\mathbf{P}_0(M>0)>0$, this gives $0<\mathbf{E}_0 P_n(M)=p_n(0)=0$, a contradiction.
\par\noindent
{\rm (b)} is a consequence of
\begin{equation}\label{eq:aux}
p_n(x)=\E P_n({}M)\geq 0 \text{ for $x\geq x^*$},
\end{equation}
as $\mathbf{P}_x\{M\geq x^*\}>0$ and $P_n(x)>0$
for $x\geq x^*$. The condition $\proba (M>x)>0$ and inequality $P_n(x)>0$ for $x>x^*$ gives the strict inequality in \eqref{eq:aux},
concluding the proof.
\end{proof}

\begin{theorem}\label{theorem:osp-polynomials}
Let $p_n(x)$ be a polynomial of degree $n$ with leading coefficient $a_n=1$ and $p_n(0)=0$.
Define as before
$$
g(x)=\left(p_n(x^+)\right)^+.
$$
Denote by $P_n(x)$ the averaging polynomial of $p_n(x)$ for the random variable $M$.
Denote by $x^*$ the largest positive root of $P_n(x)$.
Define
$
{G}(x)={P}_n(x)\indicator_{\{x\geq x^*\}},
$
and
$$
{V}(x)=\E {G}({}M),\quad \tau^*=\inf\{t\geq 0\colon {}X_t\geq x^*\}.
$$
If
$G(x)$  is non-decreasing and
$V(x)\geq g(x)$  for $x\leq x^*$,
then, the pair ${V}(x)$, $\tau^*$
is a solution of the optimal stopping problem \eqref{eq:osp}.
\end{theorem}
\begin{proof}The result follows directly from the application of Theorem \ref{theorem:one}.
\end{proof}
\section{Examples}

In order to illustrate our results we first assume that $X$   is the Brownian motion and $r=1/2$.
In this case $M$ has exponential distribution with parameter one (in the general case with parameter
$1/\sqrt{2r}).$ Its moments satisfy
$\mu_n=\Gamma(n+1)=n!$
Observe that  for spectrally negative L\'{e}vy processes, the random variable $M$ is also exponentially distributed, with parameter $\Phi(r)$, where $\Phi(r)$ is the unique positive root of equation $\psi(\lambda)=r$ and $\psi(\lambda)$ is the Laplace exponent of corresponding  L\'{e}vy process (see for instance \cite{bertoin}).
In this case we can produce  similar results.

\subsection{Example 1: Quadratic polynomials}\label{example:1}
Consider $p_2(x)=x^2+ax$. Solving \eqref{eq:system} we obtain
$$
P_2(x)=x^2+(a-2\mu_1)x+2(\mu_1)^2-\mu_2-a\mu_1
$$
that has its largest root
$$
x^*=\mu_1-a/2+\sqrt{\mu_2-\mu_1^2+a^2/4}= \Ez M-a/2+\sqrt{\var_0 M+a^2/4},
$$
that is evidently positive that can be checked independently of Proposition \ref{proposition:aux}.
In case $a=0$ we obtain
$$
x^*=\mu_1-a/2+\sqrt{\mu_2-\mu_1^2+a^2/4}= \Ez M+\sqrt{\var_0M}
$$
that gives the solution found in \cite{novikov-shiryaev:2004}. In our particular case $\mu_1=1, \mu_2=2$,
$P_2(x)=x^2+(a-2 )x -a,$ and $x^*=1-a/2+\sqrt{1+a^2/4}$. For any $a\in\R$ it is evident that $G(x)$ increases after $x^*$ and  it is not difficult to calculate $V(x)$: $$V(x)=(x^2+ax)\indicator_{\{x>x^*\}}+((x^*)^2+ax^*)e^{x-x^*}\indicator_{\{x\leq x^*\}}.$$

In order to apply Theorem \ref{theorem:osp-polynomials}, we need only to check the condition $V(x)\geq g(x)$  for $x\leq x^*$, but in fact it is only necessary to check this for $-a\leq x\leq x^*$.
Consider the case $a<0$, the opposite case is considered similarly. So, we need to check the condition
\begin{equation}\label{eq123}(x^2+ax)e^{-x}\leq ((x^*)^2+ax^*)e^{-x^*}\end{equation} for $-a\leq x\leq x^*$. The latter inequality holds  for $x=-a$ where we have the strict inequality and for $x=x^*$ where we have the equality. Furthermore, function $f(x)=(x^2+ax)e^{-x}$ has the derivative
$f'(x)=-P_2(x)e^{-x}$ which is positive between the roots of $P_2(x)$, the biggest is $x^*$.  Moreover, the smallest root of $P_2(x)$ equals $ 1-a/2-\sqrt{1+a^2/4}<-a$ for negative $a$. It means that $f(x)$ increases on $(-a,x^*)$ whence we get \eqref{eq123}. So, according to Theorem \ref{theorem:osp-polynomials}, $(V(x), x^*)$ create a solution of \eqref{eq:osp}. The same is true for $a>0$. In Fig. \ref{figure:1} we plot the solution for $a=-1$ and $a=1$.
\begin{figure}[h]\label{figure:1}
\begin{center}
\includegraphics[scale=0.45]{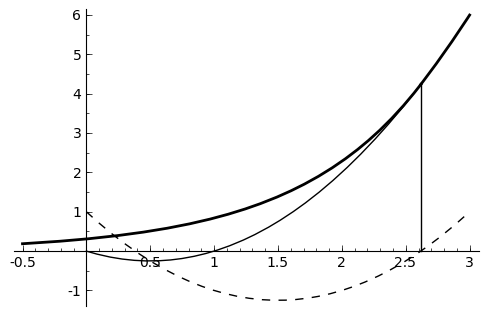}
\includegraphics[scale=0.45]{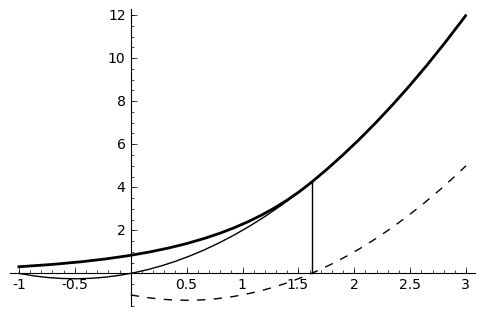}
\caption{Example 1 with $a=1$ (left) and $a=-1$ (right). Here $p_2(x)$ is plotted continuously, $P_2(x)$ (dashed) gives the roots. The thick lines are the respective solutions $V(x)$.
Observe that in case $a=1$ (left) the averaging function $P_2$ does not remain non-positive for values smaller than the root $x^*\sim 2.62$.}
\end{center}
\end{figure}
\subsection{Example 2: Cubic polynomials}
Consider $p_3(x)=x^3+ax^2+bx$. Solving \eqref{eq:system} we obtain
\begin{multline*}
P_3(x)=x^3+(a-3\mu_1)x^2+(b-3\mu_2-2(a-3\mu_1)\mu_1)x\\
-\mu_3-(a-3\mu_1)\mu_2-(b-3\mu_2-2(a-3\mu_1)\mu_1)\mu_1.
\end{multline*}
If we further assume that $X$ is the Brownian motion and $r=1/2$, we have

$$P_3(x)=x^3+(a-3)x^2+(b-2a)x -b $$
and $$V(x)=(x^3+ax^2+bx)\indicator_{\{x>x^*\}}+((x^*)^3+a(x^*)^2+bx^*)e^{x-x^*}\indicator_{\{x\leq x^*\}}.$$
If $b>0$, $P_3$ evidently has at least one positive root since $P_3(0)<0$ and $P_3(+\infty)=+\infty$. Let $b<0, a>-2$, then $P_3(1)=-a-2<0$, and at least one  positive root exceeds $1$. If  $b<0, a\leq-2$, then $P(-a)=-a^2-b(1+a)<0$  and at least one  positive root exceeds $-a$. So, in any case $P_3$ has positive roots,   in accordance with  Proposition \ref{proposition:aux} but we have checked this  independently. Now, in order to apply Theorem \ref{theorem:osp-polynomials},   consider some particular cases. In the case when $3<a<\frac{b}2\wedge 8$, (the case of positive coefficients, for example, $a=4, b=10$)  $p_3(x)$  has only one root $x=0$ because  other roots that should equal $x_{1,2}=-\frac{a}2\pm\sqrt{\frac{a^2}{4}-b}$ do not exist (discriminant is negative, since $\frac{a^2}{4}-b<\frac{ab}{8}-b<0$), the derivative $P_3'(x)=3x^2+2(a-3)x+b-2a$ has two negative roots therefore it is positive  on $[0,\infty)$, and $P_3(x)$ increases on $[x^*, \infty)$, even more, it increases on $[0, \infty)$ being negative on $[0, x^*)$. Moreover, as in the example with quadratic polynomials, we need to check inequality
$$(x^3+ax^2+bx)e^{-x}\leq ((x^*)^3+a(x^*)^2+bx^*)e^{-x^*}$$
on the interval $[0, x^*]$  but on the interval $[0, x^*]$ the derivative of the function $(x^3+ax^2+bx)e^{-x}$ being equal $-P_3(x)e^{-x}$ is positive therefore both conditions of Theorem \ref{theorem:osp-polynomials} hold.
In the  case $a=b=0$, i.e. $p_3(x)=x^3$, we have

\begin{equation*}
P_3(x)=x^3-3x^2,
\end{equation*}
with largest root $x^*=3$ (see Fig. \ref{figure:3}). Evidently, $P_3$ increases on $[x^*, \infty)$ because its derivative $3x^2-6x$ is positive on the interval  $[2, \infty)$ and $-P_3'$ is positive on $(0,3)$ which supplies both conditions of Theorem \ref{theorem:osp-polynomials}.
\begin{figure}[h]
\begin{center}
\includegraphics[scale=0.45]{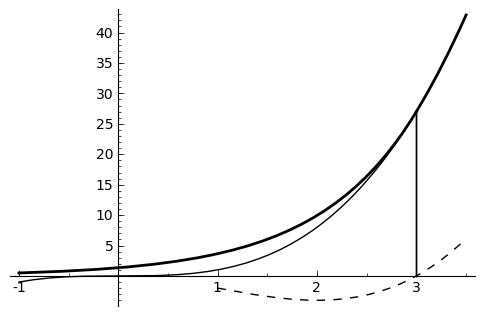}
\includegraphics[scale=0.45]{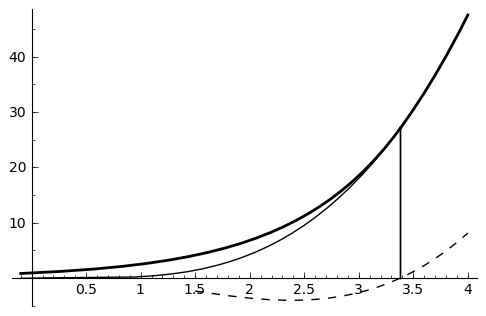}
\caption{Example 2 with $a=b=0$ (left) and $a=-9/8$, $b=3/8$ (left). Here $P_3(x)$ is plotted continuously,
$Q_3(x)$ (dashed) gives the root $x^*$ in each case. The thick lines are the respective solutions $V(x)$.}
\label{figure:2}
\end{center}
\end{figure}
An example for Brownian motion with $r=1/2$ and   polynomial with positive $b$ and negative $a$
is shown in Fig. \ref{figure:2}. We put in this case
$$
p_3(x)=x^3-(9/8)x^2+3/8,\qquad
P_3(x)=x^{3} - \frac{33}{8} \, x^{2} + \frac{21}{8} \, x - \frac{3}{8},
$$
and the largest root is $x^*=3.3815$.

\subsection{Example 3: Kou's process for a quartic polynomial}
A diffusion process $\{X_t\}$ with two sided exponential jumps, defined by the formula
$$
X_t=at+\sigma W_t+\sum_{k=1}^{N_t}Y_k-\sum_{k=1}^{N'_t}Y'_k,
$$
is known in the financial literature as a Kou's process (see \cite{kou} and \cite{ct}).
Here $\{N_t\}$ (resp. $\{N'_t\}$) is a Poisson process with parameter $\mu$ (resp. $\nu$)
and $\{Y_k\}$ (resp $\{Y'_k\}$) is a sequence of independent exponential random variables with parameter
$\alpha$ (resp $\beta$).
The charactristic exponent \eqref{eq:char-exp} of the process
is given by
$$
\psi(z)=az+\frac12\sigma^2z^2+\mu\frac{z}{\alpha-z}-\nu\frac{z}{z+\beta},
$$
and the density of the maximum $M$ in this case is a mixture of two exponentials
$$
f_M(x)=A_1r_1e^{-r_1x}+A_2r_2e^{-r_2x},
$$
with coefficients
$$
A_1={1-r_1/\beta\over 1-r_1/r_2},\quad
A_2={1-r_2/\beta\over 1-r_2/r_1},
$$
where $0<r_1<r_2$ are the positive roots of the equation $\psi(z)=r$
(see \cite{mordecki:2003}).  In consequence the moments are given by
$$
\mu_k=k!\left(\frac{A_1}{r_1^k}+\frac{A_2}{r_2^k}\right).
$$
We consider a quartic polynomial
$$
p(x)=x^{4} + a_3x^{3} + a_2x^{2} +a_1x.
$$
If we denote $P(x)=x^4+b_3x^3+b_2x^2+b_1x+b_0$,
applying \eqref{eq:system} we obtain
\begin{align*}
b_3&=a_3-4\mu_1\\
b_2&=a_2-6\mu_2-3b_3\mu_1\\
b_1&=a_1-4\mu_3-3b_3\mu_2-2b_2\mu_1\\
b_0&=-(\mu_4+b_3\mu_3+b_2\mu_2+b_1\mu_1)
\end{align*}
Assuming that there exists a value $x^*$ that satisfies the conditions of Theorem \ref{theorem:osp-polynomials},
we write the possible value function
\begin{equation*}
V(x)=\E P(x+M)\indicator_{\{x+M\geq x^*\}}=B_1e^{r_1(x-x^*)}+B_2e^{r_2(x-x^*)},
\end{equation*}
where
$$
B_1=A_1r_1\int_0^\infty P(z+x^*)e^{-r_1z}dz,\quad
B_2=A_2r_2\int_0^\infty P(z+x^*)e^{-r_2z}dz.
$$
To proceed we choose values for the parameters:
$$
a=2,\ \sigma=1,\ \mu=1,\ \nu=1,\ \alpha=2,\ \beta=2,\ r=6.
$$
and choose the polynomial
$$
p(x)=x(x-1)(x-2)(x-3)=x^{4} - 6x^{3} + 11x^{2} - 6x.
$$
We obtain $r_1=1.4327$, $r_2=2.8740$, giving $A_1=0.5656,$ $A_2=0.4344$.
In consequence $x^*=4.3706$. For this sets of parameters the conditions of Theorem \ref{theorem:osp-polynomials} are fulfilled (see Figure \ref{figure:3}).
\begin{figure}[h]
\begin{center}
\includegraphics[scale=0.45]{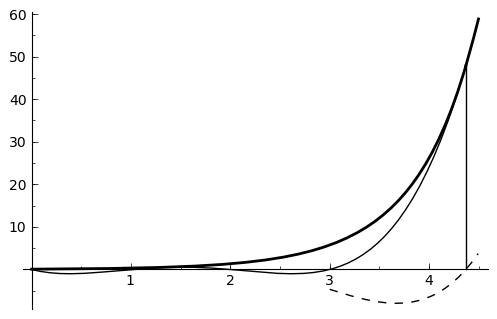}
\includegraphics[scale=0.45]{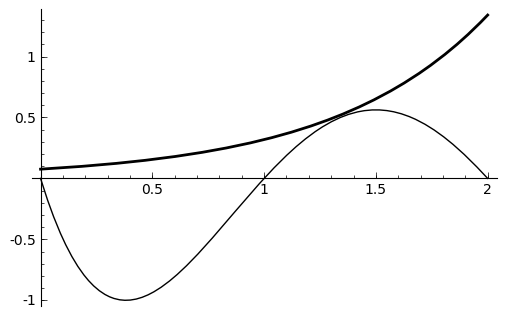}
\caption{On the left $p(x)$ is plotted continuously, $P(x)$ (dashed) gives the root $x^*=4.37
$. The thick line is the solution $V(x)$.
On the right we observe that $V(x)\geq g(x)$.}
\label{figure:3}
\end{center}
\end{figure}

\end{document}